\documentclass[12pt]{article}

\usepackage[utf8]{inputenc}
\usepackage{hyperref}
\usepackage{amsmath}
\usepackage{eucal}
\usepackage{amssymb}
\usepackage{stmaryrd}
\usepackage{amsthm}
\usepackage{tikz-cd}
\usepackage{mathtools}
\usepackage{xcolor}
\usepackage{geometry}

\usepackage{mathbbol}
\usepackage{amsfonts}
\DeclareSymbolFontAlphabet{\mathbb}{AMSb}
\DeclareSymbolFontAlphabet{\mathbbl}{bbold}

\def\KK{{\mathbbl{k}}} 
\def\N{{\mathbb{N}}}  
\def\R{{\mathbb{R}}}  
\def\C{{\mathbb{C}}}  
\def\LL{{\mathbb{L}}}  
\def\A{{\mathcal{A}}} 
\def\E{{\mathcal{E}}} 
\def\F{{\mathcal{F}}} 
\def\GG{{\mathbb{G}}} 
\def\O{{\mathcal{O}}} 
\def\S{{\mathcal S}} 
\def\X{{\mathcal X}} 
\def\M{{\mathcal M}} 
\def\N{{\mathcal N}} 

\let\LL\relax
\def\LL{{\mathbb L}}

\DeclareMathOperator{\Spec}{Spec}
\DeclareMathOperator{\Sym}{Sym}
\DeclareMathOperator{\SCS}{\S_{CS}}
\DeclareMathOperator{\Ker}{Ker}

\DeclareMathOperator{\id}{id}
\DeclareMathOperator{\Qcoh}{Qcoh}
\DeclareMathOperator{\Coh}{Coh}
\DeclareMathOperator{\CohM}{\mathcal C oh_\M}
\DeclareMathOperator{\VecM}{\mathcal V ec_\M}
\let\H\undefined
\DeclareMathOperator{\H}{H}

\DeclareMathOperator{\h}{h}

\DeclareMathOperator{\Hom}{Hom}

\DeclareMathOperator{\Aut}{Aut}

\DeclareMathOperator{\D}{D}

\DeclareMathOperator{\obj}{obj}
\DeclareMathOperator{\Ext}{Ext}
\DeclareMathOperator{\Left}{L}
\DeclareMathOperator{\BGL}{BGL}
\DeclareMathOperator{\GL}{GL}

\newtheorem{thm}{Theorem}[section]
\newtheorem{prop}[thm]{Proposition}
\newtheorem{lemma}[thm]{Lemma}
\newtheorem{coro}[thm]{Corollary}
\theoremstyle{definition}
\newtheorem{defi}[thm]{Definition}
\newtheorem{notation}[thm]{Notation}
\theoremstyle{remark}
\newtheorem{remx}[thm]{Remark}
\newenvironment{rem}
  {\pushQED{\qed}\remx}
  {\popQED\endremx}

\newtheoremstyle{TheoremNum}
        {\topsep}{\topsep}              
        {\itshape}                      
        {}                              
        {\bfseries}                     
        {.}                             
        { }                             
        {\thmname{#1}\thmnote{ \bfseries #3}}
\theoremstyle{TheoremNum}
\newtheorem{thmn}{Theorem}

\begin{document}

\title{A perfect obstruction theory for moduli of coherent systems}
\author{Giorgio Scattareggia}
\date{}

\maketitle

\section*{Abstract}
Let $C$ be a curve of genus $g$. A \textit{coherent system} on $C$ is a pair $(E,V)$, where $E$ is a finite rank vector bundle on $C$ and $V$ is a linear subspace of the space of global sections of $E$. The \textit{type} of a coherent system $(E,V)$ is a triple $(n,d,k)$, where $n$ is the rank of $E$, $d$ is the degree of $E$ and $k$ is the dimension of $V$. The notion of stability for a coherent system $(E,V)$ differs from the stability of the bundle $E$ and depends on the choice of a real parameter $\alpha$. The moduli space of $\alpha$-stable coherent systems of type $(n,d,k)$ has an expected dimension $\beta = \beta(n,d,k)$ which depends on the genus of the curve $C$ and on the type of the coherent systems.

We construct a perfect obstruction theory for the moduli spaces of $\alpha$-stable coherent systems which has rank equal to the expected dimension $\beta$. In our construction we do not fix one curve, but we work on families of projective Gorenstein curves.

\section*{Introduction}

In the early 1990s Le Potier et al introduced the definition of coherent systems, shortly $CS$, in order to generalize the classical notion of linear series for higher rank vector bundles \cite{LePotier}, \cite{Bertram}, \cite{raghavendra}. A coherent system on a smooth curve $C$ is a pair $(E,V)$, where $E$ is a finite rank vector bundle on $C$ and $V$ is a subspace of the vector space of global sections of $E$. We say that $(E,V)$ is a coherent system of type $(n, d, k)$ if $n$ is the rank of $E$, $d$ is the degree of $E$ and $k$ is the dimension of $V$; a coherent system of type $(1,d,k)$ is a linear series $g^{k-1}_d$.

In 1995 King and Newstead introduced a notion of (semi-)stability for coherent systems, which depends on the choice of a real parameter $\alpha$, and they constructed the moduli spaces of $\alpha$-stable coherent systems as GIT quotients \cite{KingNew}. The choice of the parameter $\alpha$ is equivalent to the choice of a GIT linearization.

In 1998 He proved that every moduli space of $\alpha$-stable coherent systems of type $(n,d,k)$ on a smooth genus $g$ curve has expected dimension $\beta = \beta(n,d,k) =  n^2(g - 1) + 1 - k(k - d + n(g - 1))$ \cite{He}. The integer $\beta$ is called the Brill Noether number and it reduces to the usual Brill Noether number defined for linear series, $\rho = \beta(1,d,k) = g - k(g - d+(k-1))$, if $n=1$ \cite[2.7]{CohBrNoet}.

In this paper we construct a perfect obstruction theory (in the sense of \cite{BFantechi}) for the moduli spaces of $\alpha$-stable coherent systems which justifies their expected dimension.

Since only some of the hypotheses in the definition of coherent systems are relevant for this construction, we work in a more general setting. Indeed, we introduce the notion of generalized coherent systems, shortly $GCS$, which relaxes some of the properties of a coherent system. A $GCS$ on a projective curve $C$ is a triple $(F, E, \varphi)$ where $F \in \Coh(C)$, $E$ is a finite rank vector bundle on $C$ and $\varphi : F \rightarrow E$ is a morphism of sheaves. Every coherent system $(E, V)$ on a smooth curve $C$ naturally induces a generalized coherent system $(V \otimes \O_C, E, \varphi)$ on $C$, where the map $\varphi : V \otimes \O_C \rightarrow E$ is determined by the injection $V \subseteq \H^0(C, E)$. We fix a flat family of Gorenstein projective curves over an algebraic stack $\M$ and we construct the moduli stack $\S$ of families of $GCS$ on curves in $\M$. Then we prove that $\S$ is algebraic in the sense of Artin. We also point out that $\S$ comes with a natural representable forgetful morphism $G : \S \rightarrow \N$, where $\N := \CohM \times_\M \VecM$ is the moduli stack of pairs $(F,E)$.

The central result of this paper is the construction of a perfect relative obstruction theory for the forgetful morphism $G : \S \rightarrow \N$.
\begin{thmn}[\ref{PobTh}]
There is a  canonical morphism
\[
E^\bullet \longrightarrow \tau_{\geq -1}\Left_G
\]
which is a perfect relative obstruction theory for the forgetful morphism $G : \S \rightarrow \N$.
\end{thmn}
The construction of this relative obstruction theory allows us to define a perfect obstruction theory for the moduli spaces of simple coherent systems (i.e.\ $CS$ whose group of automorphisms is the scalars). As $\alpha$-stable coherent systems are simple, such an obstruction theory induces a perfect obstruction theory for every moduli space of $\alpha$-stable coherent systems.

\begin{thmn}[\ref{CobStabCohSys}]
Fix $\alpha \in \R$. Let $C$ be a smooth, projective, genus $g$ curve and let $(n,d,k)$ be a suitable triple of positive integers. Let $\beta := \beta(n,d,k)$ be the Brill Noether number. Then the moduli space of $\alpha$-stable coherent systems of type $(n,d,k)$ has a perfect obstruction theory of rank~$\beta$.
\end{thmn}

\section*{Acknowledgments}
These pages represent the conclusion of a wonderful experience that has been my Ph.D at SISSA (my Ph.D thesis is available online in the SISSA digital library \cite{Tesi}). It is very important for me to use some words in order to thank both my supervisors, Barbara Fantechi and Fabio Perroni, for all their support and their guide. Thanks to you and to all the people in SISSA I have spent four astonishing amazing years!

\bigskip
\bigskip
\noindent \textit{Notations and conventions}

\bigskip
\noindent Unless otherwise mentioned, we work over an algebraically closed field $\KK$.

We denote by $(Sch)$ the category of schemes of finite type over $\KK$; we assume all schemes to be objects of $(Sch)$.

A \textit{groupoid} is a category in which every morphism is an isomorphism.

A \textit{category fibered in groupoids} is a category fibered in groupoids over $(Sch)$ in the sense of~\cite{Olsson}.

If $T$ is a scheme and $X$ is a category fibered in groupoids, a \textit{$T$-point} of $X$ is an object in the groupoid $X(T)$.

An \textit{algebraic stack}, or Artin stack, is an algebraic stack over $\KK$ in the sense of~\cite{Artin1974} or~\cite{Olsson}. We assume all algebraic stacks to be locally of finite type over $\KK$.

A \textit{Deligne Mumford stack} is a Deligne-Mumford stack in the sense of~\cite{Deligne1969} or \cite{Olsson}.

A \textit{coherent system} is a coherent system in the sense of~\cite{LePotier} or equivalently a Brill-Noether pair in the sense of \cite{KingNew}.


\section{Generalized coherent systems}\label{Sstacks}

Recall that a coherent system of type $(n, d, k)$ on a smooth curve $C$ is a pair $(E,V)$, where $E$ is a vector bundle on $C$ of rank $n$ and degree $d$, and $V$ is a linear subspace of $H^0(C,E)$ of dimension $k$ \cite{LePotier}, \cite{KingNew}.

The following definition generalizes the notion of coherent system and it is the central object of this work.

\begin{defi}
Let $C$ be a projective curve over $\KK$. A \textit{generalized coherent system} on $C$ is a morphism of sheaves $\varphi : F \rightarrow E$, where $F, E \in \Coh(C)$ and $E$ is locally free.
\end{defi}

Notice that a coherent system $(E,V)$ on a smooth projective curve $C$ naturally induces a generalized coherent system $V \otimes \O_C \rightarrow E$.

In this section we construct a moduli stack of generalized coherent systems and we prove that it is an algebraic stack in the sense of Artin. 

\bigskip
Fix a ground algebraic stack $\M$ together with a flat projective relatively Gorenstein morphism $\M' \rightarrow \M$ of relative dimension 1. One may assume either that $\M = \M_g$ is the algebraic stack of genus $g$ smooth curves, or that $\M = \overline \M_g$ is the Deligne Mumford compactification of $\M_g$. We let $\M$ be any algebraic stack satisfying these properties, since we intend to work in the greatest possible generality.

\begin{notation}\label{NfamCurves}
Let $T$ be a scheme. If $T \rightarrow \M$ is a morphism of stacks, then $\M' \times_\M T \rightarrow T$ is a morphism of schemes and it is a relatively Gorenstein flat projective family of curves over $T$. On the other hand, if we say that $C \rightarrow T$ is a \textit{family of curves} (or more specifically an \textit{$\M$-family of curves} over $T$) we always mean that we have fixed a morphism of stacks $T \rightarrow \M$, $C = \M' \times_\M T$ and $C \rightarrow T$ is the second projection. Hence all the families of curves that we consider are projective and relatively Gorenstein.
\end{notation}


\begin{defi}\label{Dgcs}
Let $T$ be a scheme. A \textit{family of generalized coherent systems} $(C \rightarrow T, F \rightarrow E)$ over $T$ is defined by the following data:
\begin{enumerate}
\item an $\M$-family of curves $C \rightarrow T$ (as in \ref{NfamCurves});
\item a morphism of sheaves $F \rightarrow E$, where $F, E \in \Coh(C)$, $F$ is flat over $T$ and $E$ is locally free.
\end{enumerate}

An \textit{isomorphism} $(C \rightarrow T, F \rightarrow E) \rightarrow (C' \rightarrow T, F' \rightarrow E')$ of families of generalized coherent systems over $T$ is a triple $(\alpha, \beta, \gamma)$, where $\alpha : C \rightarrow C'$ is induced by an isomorphism in $\M(T)$, $\beta : F \rightarrow \alpha^*F'$ and $\gamma : E \rightarrow \alpha^*E'$ are isomorphisms of sheaves such that the following diagram commutes:
\[
\begin{tikzcd}
F \arrow[r] \arrow[d] & \alpha^*F'  \arrow[d] \\
E \arrow[r] & \alpha^*E' \,.
\end{tikzcd}
\]
\end{defi}

Let us denote the groupoid of families of generalized coherent systems over $T$ by $\S(T)$. Letting $T$ vary, we get a category fibered in groupoids that we denote by $\S$. Actually it is an algebraic stack, as we will prove in \ref{CstrongRep}. We call $\S$ the \textit{moduli stack of generalized coherent systems}. It comes together with a natural forgetful morphism
\[
\S \rightarrow \M \,.
\]
Notice that the induced groupoid functor $\S(T) \rightarrow \M(T)$ is not faithful, as there exist non trivial isomorphisms in $\S(T)$ which map to the identity in $\M(T)$.

\begin{notation}\label{NunivFamS}
Let $\S' := \M' \times_\M \S$ and let $\bar \pi : \S' \rightarrow \S$ denote the second projection.

The morphism $\bar \pi : \S' \rightarrow \S$ is relatively Gorenstein and we denote by $\bar \omega$ its dualizing bundle.

The stack $\S$ has a universal family which is a morphism $\phi : \bar \F \rightarrow \bar \E$ in $\Coh(\S')$. Notice that $\bar \F$ is flat over $\S$ and $\bar \E$ is locally free.
\end{notation}

\begin{defi}\label{DstackN}
Let $T$ be a scheme. Define a groupoid $\CohM(T)$ such that:
\begin{enumerate}
\item the objects of $\CohM(T)$ are pairs $(C \rightarrow T, F)$, where $C \rightarrow T$ is an $\M$-family of curves~(as in~\ref{NfamCurves}) and $F \in \Coh(C)$ is flat over $T$;
\item the isomorphisms $(C \rightarrow T, F) \rightarrow (C' \rightarrow T, F')$ of $\CohM(T)$ are pairs $(\alpha, \beta)$, where $\alpha : C \rightarrow C'$ is induced by an isomorphism in $\M(T)$ and $\beta : F \rightarrow \alpha^*F'$ is an isomorphism of sheaves.
\end{enumerate}
Letting $T$ vary, we get a category fibered in groupoids that we denote by $\CohM \,$. It comes with a natural forgetful morphism $\CohM \rightarrow \M$.

Analogously, define a category fibered in groupoids $\VecM$ whose $T$-points are pairs $(C \rightarrow T, E)$, where $C \rightarrow T$ is an $\M$-family of curves and $E \in \Coh(C)$ is locally free.

Define
\[
\N := \CohM \times_\M \VecM \,.
\]
Notice that there is a natural forgetful morphism $\N \rightarrow \M$. 
\end{defi}

\begin{lemma}
The category fibered in groupoids $\N$ is an algebraic stack.
\end{lemma}

\begin{proof}
It suffices to show that $\CohM$ and $\VecM$ are algebraic stacks. This is a standard result, see for example \cite{Olsson} or \cite{Hei2010}. 
\end{proof}

There is a natural forgetful morphism
\[
G : \S \rightarrow \N \,.
\]
The induced groupoid functor $\S(T) \rightarrow \N(T)$ is faithful, as every isomorphism in $\S(T)$ comes from an isomorphism in $\N(T)$. Hence the morphism $G : \S \rightarrow \N$ is representable. In \ref{PstrongRep} we shall see that $G : \S \rightarrow \N$ is actually strongly representable and it gives $\S$ the structure of an abelian cone over $\N$.

\begin{notation}\label{NunivFamN}
Let $\N' := \M' \times_\M \N$ and let $\pi : \N' \rightarrow \N$ denote the second projection. 

The morphism $\pi : \N' \rightarrow \N$ is relatively Gorenstein and we denote by $\omega$ its dualizing bundle.

The stack $\N$ has a universal family which consists of a pair of coherent sheaves $(\F, \E)$ on $\N'$. Notice that $\F$ is flat over $\N$ and $\E$ is locally free.

In \ref{NunivFamS} we denoted the universal family of $\S$ by $\phi : \bar \F \rightarrow \bar \E$ and the dualizing bundle of the morphism $\bar \pi : \S' \rightarrow \S$ by $\bar \omega$. Notice that $\S' \cong \N' \times_\N \S$; let $\bar G : \S' \rightarrow \N'$ denote the first projection, then $\bar \F \cong \bar G^* \F$, $\bar \E \cong \bar G^* \E$ and $\bar \omega \cong \bar G^* \omega$.
\end{notation}

\begin{prop}\label{PstrongRep}
There is a natural isomorphism of $\N$-stacks
\[
\S \rightarrow \Spec\Sym(\R^1\pi_*(\F \otimes \E^\vee \otimes \omega)) \,.
\]
In particular, $\S$ is an abelian cone over $\N$.
\end{prop}

\begin{proof}
Fix a scheme $T$. Recall that $\S(T) = \{t : T \rightarrow \N, \varphi : \bar t^* \F \rightarrow \bar t^* \E\}$, where we have used the notation $\bar t := (t)_\pi : \N' \times_\N T \rightarrow \N'$. Define $\A := \Spec\Sym(\R^1\pi_*(\F \otimes \E^\vee \otimes \omega))$; then $\A(T) = \{t : T \rightarrow \N, \gamma : t^*\R^1\pi_*(\F \otimes \E^\vee \otimes \omega) \rightarrow \O_T\}$. By Grothendieck duality and cohomology and base change there is a canonical bijection $\Hom(\bar t^* \F, \bar t^* \E) \rightarrow  \Hom(t^*\R^1\pi_*(\F \otimes \E^\vee \otimes \omega), \O_T)$ which induces an equivalence of groupoids $\S(T) \rightarrow \A(T)$. This equivalence is compatible with pullbacks and hence it induces an isomorphism of stacks $\S \rightarrow \A$. Notice that it is indeed an isomorphism of $\N$-stacks.
\end{proof}

\begin{coro}\label{CstrongRep}
The moduli stack of generalized coherent systems $\S$ is an algebraic stack and the forgetful morphism $G : \S \rightarrow \N$ is strongly representable. \qed
\end{coro}

Now we show that the morphism $G : \S \rightarrow \N$ locally factorizes as the composition of a smooth morphism and a closed embedding. We need the following preliminary result.

\begin{lemma}\label{LglobRes}
Let $p : X \rightarrow Y$ be a flat projective morphism of algebraic stacks of relative dimension 1; assume that $Y$ is quasi-compact. Let $F \in \Coh(X)$ be a flat sheaf over $Y$. Then $F$ has a resolution
\[
0 \rightarrow K \rightarrow M \rightarrow F \rightarrow 0
\]
where $M$ is locally free, $p_* K = p_* M = 0$ and $\R^1p_* K$ and $\R^1p_* M$ are locally free. We say that $0 \rightarrow K \rightarrow M \rightarrow F \rightarrow 0$ is a resolution of $F$ with respect to $p : X \rightarrow Y$.
\end{lemma}

\begin{proof}
We split the proof in two parts.

$(i)$ Assume that $p : X \rightarrow Y$ is a morphism of schemes and that $Y$ is an affine scheme.
Let $n \gg 0$ and $M := p^*p_*(F \otimes \O_X(n)) \otimes \O_X(-n)$. The sheaf $M$ is locally free, since $F$ is flat over $Y$ \cite[III.9.9]{hartshorne};  the canonical morphism $M \rightarrow F$ is a surjection \cite[III.8.8]{hartshorne}. For every point $y \in Y(\KK)$ let $X_y := X \times_Y \Spec \KK$ and let $\bar y : X_y \rightarrow X$ be the first projection; since $\KK$ is algebraically closed and $Y$ is quasi-compact, we have that $\H^0(X_y, \bar y^* M) = 0$ for every $y \in Y(\KK)$ \cite[III.7.6 and III.12.11]{hartshorne}. By Nakayama's Lemma that implies that $p_* M = 0$. In particular, $\R^1 p_* M$ is locally free.

Let $K := \Ker(M \rightarrow F)$; the sheaf $K$ is flat over $Y$ since $F$ is flat over $Y$, $M$ is locally free and $p : X \rightarrow Y$ is a flat morphism of schemes. Moreover $p_* K = 0$. Hence $\R^1p_* K$ is locally free.

$(ii)$ One can check that the previous construction works for morphisms of algebraic stacks using descent for coherent sheaves.
%
%
%
\end{proof}

\begin{rem}\label{RglobRes}
With the notation introduced in~\ref{LglobRes} fix a resolution $0 \rightarrow K \rightarrow M \rightarrow F \rightarrow 0$ of $F$ with respect to $p : X \rightarrow Y$. This resolution behaves well with base change, meaning that if $q : Z \rightarrow Y$ is a morphism of algebraic stacks and $\bar q : X \times_Y Z \rightarrow X$ denotes the first projection,  then $0 \rightarrow \bar q^*K \rightarrow \bar q^*M \rightarrow \bar q^*F \rightarrow 0$ is an exact sequence and it is a resolution of $\bar q^* F$ with respect to $X \times_Y Z \rightarrow Z$ (in the sense of~\ref{LglobRes}).
\end{rem}

\begin{prop}\label{PfactorizOfG}
Locally on $\N$, the forgetful morphism $G : \S \rightarrow \N$ factorizes as the composition of a closed embedding followed by a smooth morphism.
\end{prop}

\begin{proof}
By restricting to an open subset, we may assume that $\N$ is quasi-compact. Let $F := \F \otimes \E^\vee \otimes \omega \in \Coh(\N')$ (we set the notation in \ref{NunivFamN}). Choose a resolution $0 \rightarrow K \rightarrow M \rightarrow F \rightarrow 0$ of $F$ with respect to $\pi : \N' \rightarrow \N$ (as in \ref{LglobRes}). According to Proposition~\ref{PstrongRep}, the surjective morphism $M \rightarrow F$ induces a closed embedding $\S \rightarrow \Spec\Sym(\R^1\pi_*M)$ which is a morphism of $\N$-stacks. Since $\R^1\pi_*M$ is locallly free, the structure morphism $\Spec\Sym(\R^1\pi_*M) \rightarrow \N$ is smooth.
\end{proof}

We can use this factorization of the morphism $G : \S \rightarrow \N$ in order to give a local description of the truncated cotangent complex of $G$. 
That will be useful for the construction of a perfect obstruction theory for $G$. Indeed, as we will see in the proof of Proposition~\ref{PobTh}, it is essential to have a local description of the cotangent complex of $G$ related the resolutions introduced in~\ref{LglobRes}.

\begin{coro}\label{CcotcmpxFact}
Choose a local factorization $\S \rightarrow X \rightarrow \N$ of $G : \S \rightarrow \N$, as in~\ref{PfactorizOfG}. Let $I$ be the ideal sheaf of $\S \rightarrow X$ and $\Omega$ be the cotangent bundle of $X \rightarrow \N$. Then $\tau_{\geq -1} \Left_G \cong [I_{|\S} \rightarrow \Omega_{|\S}]$.
\end{coro}

\begin{proof}
That is a consequence of \cite[Tag 08SH]{stacks-project}.
\end{proof}

\section{The obstruction theory}\label{SobTh}

In this section we construct a relative perfect obstruction theory for the forgetful morphism $G : \S \rightarrow \N$ (defined in Section~\ref{Sstacks}). In order to fix the notation, we recall the definition of obstruction theory as it is introduced in \cite{BFantechi}.

\begin{defi}
Let $X \rightarrow Y$ be a Deligne Mumford morphism of algebraic stacks; let $E^\bullet \in \obj D^{[-1,0]}_{\Coh}(X)$; let $\Left_{X/Y}$ be the cotangent complex of $X \rightarrow Y$. A morphism
\[
\xi : E^\bullet \longrightarrow \tau_{\geq -1}\Left_{X/Y}
\]
in $D^{[-1,0]}_{\Coh}(X)$ is called an \textit{obstruction theory} for the morphism $X \rightarrow Y$ if $h^0(\xi)$ is an isomorphism and $h^{-1}(\xi)$ is surjective.

We say that an obstruction theory $E^\bullet \rightarrow \tau_{\geq -1}\Left_{X/Y}$ is \textit{perfect}, if the complex $E^\bullet$ is perfect of perfect amplitude contained in $[-1,0]$.
\end{defi}

\begin{defi}
Let $X$ be an algebraic stack; let $E^\bullet \in \obj D^{[-1,0]}_{\Coh}(X)$ and let $M^\bullet = [M^{-1} \rightarrow M^0]$ be a morphism of locally free sheaves considered as an object of $D^{[-1,0]}_{\Coh}(X)$. An isomorphism $M^\bullet \rightarrow E^\bullet$ in $D^{[-1,0]}_{\Coh}(X)$ is called a \textit{global resolution} of~$E^\bullet$.
\end{defi}

\begin{prop}\label{PglobRes}
Let $p : X \rightarrow Y$ be a flat projective morphism of algebraic stacks of relative dimension 1. Let $F \in \Coh(X)$ be a flat sheaf over $Y$. Then $\R p_*F[1] \in D^{[-1,0]}_{\Coh}(X)$ is a perfect complex. Furthermore, if $Y$ is quasi-compact then $\R p_*F[1]$ has a global resolution.
\end{prop}

\begin{proof}
By resticting to an open subsect, we may assume that $Y$ is quasi-compact. Choose a resolution $0 \rightarrow K \rightarrow M \rightarrow F \rightarrow 0$ of $F$ with respect to $p : X \rightarrow Y$ (as in~\ref{LglobRes}). Then $[\R^1p_*K \rightarrow \R^1p_*M]$ is a global resolution of $\R p_*F[1]$.
\end{proof}

Recall that $\bar \pi : \S' \rightarrow \S$ denotes the universal curve over $\S$, $\phi : \bar \F \rightarrow \bar \E$ denotes the universal family of $\S$ and $\bar \omega$ denotes the dualizing bundle of $\bar \pi : \S' \rightarrow \S$ (see~\ref{NunivFamS} and~\ref{NunivFamN}).

\begin{prop}\label{PobTh}
There is a canonical morphism 
\[
\R \bar \pi_*(\bar \F \otimes \bar \E^\vee \otimes \bar \omega[1]) \longrightarrow \tau_{\geq -1}\Left_G
\]
which is a perfect obstruction theory for $G: \S \rightarrow \N$.
\end{prop}

\begin{proof}
We split the proof in two parts.

$(i)$ By restricting to a local chart, we may assume that $\N$ is a (quasi-compact) scheme.

With the notation of $\ref{NunivFamN}$ let $F := \F \otimes \E^\vee \otimes \omega \in \Coh(\N')$. Choose a resolution $0 \rightarrow K \rightarrow M \rightarrow F \rightarrow 0$ of $F$ with respect to $\pi : \N' \rightarrow \N$ (as in~\ref{LglobRes}). It induces a factorization $G = q \circ \iota$ where $\iota : \S \rightarrow \Spec\Sym(\R^1\pi_*M)$ is a closed embedding and $q : \Spec\Sym(\R^1\pi_*M) \rightarrow \N$ is smooth (see the proof of~\ref{PfactorizOfG}). Let $I$ denote the ideal sheaf of $\iota$ and $\Omega$ denote the cotangent sheaf of $q$. Then $\tau_{\geq -1}\Left_G \cong [I_{|\S} \rightarrow \Omega_{|\S}]$ (see~\ref{CcotcmpxFact}), we have a natural surjection $q^*\R^1\pi_*K \rightarrow I$ and a natural isomorphism $q^*\R^1\pi_*M \rightarrow \Omega$, and the following diagram commutes:
\[
\begin{tikzcd}
G^*\R^1\pi_* K \arrow[r] \arrow[d] & G^*\R^1\pi_* M \arrow[d] \\
I_{|\S} \arrow[r] & \Omega_{|\S} \,.
\end{tikzcd}
\]
Hence $[G^*\R^1\pi_* K \rightarrow G^*\R^1\pi_* M] \rightarrow \tau_{\geq -1}\Left_G$ is an obstruction theory of $G$. Recall that $\S' \cong \N' \times_\N \S$ and that $\bar G : \S' \rightarrow \N'$ denotes the first projection (see~\ref{NunivFamN}). By cohomology and base change we have that $G^*\R^1\pi_* K \cong \R^1 \bar \pi_*\bar G^* K$ and $G^*\R^1\pi_* M \cong \R^1 \bar \pi_*\bar G^* M$. Notice that $0 \rightarrow \bar G^* K \rightarrow \bar G^* M \rightarrow \bar G^* F \rightarrow 0$ is a resolution of $\bar G^* F \cong \bar \F \otimes \bar \E^\vee \otimes \bar \omega$ with respect to $\bar \pi : \S' \rightarrow \S$ (see~\ref{RglobRes}). Hence $[\R^1 \bar \pi_*\bar G^* K \rightarrow \R^1 \bar \pi_*\bar G^* M]$ is a global resolution of $\R \bar \pi_*(\bar \F \otimes \bar \E^\vee \otimes \bar \omega[1])$ (see~\ref{PglobRes}), the morphism $[\R^1 \bar \pi_*\bar G^* K \rightarrow \R^1 \bar \pi_*\bar G^* M] \rightarrow \tau_{\geq -1}\Left_G$ induces a morphism $\R \bar \pi_*(\bar \F \otimes \bar \E^\vee \otimes \bar \omega[1]) \rightarrow \tau_{\geq -1}\Left_G$ which does not depend on the choice of the resolution and is a perfect obstruction theory for $G : \S \rightarrow \N$.

$(ii)$ Since there is not a straightforward procedure to glue morphisms in the derived category of $\S$, we need to introduce some formal tools described in \cite{BFantechi}. Let $N$ be a (quasi-compact) scheme and $N \hookrightarrow \N$ be a local chart of $\N$; let $S := \S \times_\N N$ and $G_0 : S \rightarrow N$ be the second projection; let $S' := \S' \times_\S S$ and $\bar \pi_0 : S' \rightarrow S$ be the second projection; let $\bar \F_0$, $\bar \E_0$ and $\bar \omega_0$ be respectively the restriction of $\bar \F$, $\bar \E$ and $\bar \omega$ on the scheme $S'$. In the first part of the proof we have constructed a perfect obstruction theory $\xi_0 : \R (\bar \pi_0)_*(\bar \F_0 \otimes \bar \E_0^\vee \otimes \bar \omega_0[1]) \longrightarrow \tau_{\geq -1}\Left_{G_0}$ in $\D^{[-1,0]}(S)$, and we have produced a global resolution $[\R^1 (\bar \pi_0)_*\bar K_0 \rightarrow \R^1 (\bar \pi_0)_*\bar  M_0]$ of $\R (\bar \pi_0)_*(\bar \F_0 \otimes \bar \E_0^\vee \otimes \bar \omega_0[1])$. The perfect complex $[\R^1 (\bar \pi_0)_*\bar  K_0 \rightarrow \R^1 (\bar \pi_0)_*\bar M_0]$ naturally induces a vector bundle stack $[\Spec\Sym(\R^1 (\bar \pi_0)_*\bar K_0)/\Spec\Sym(\R^1 (\bar \pi_0)_*\bar M_0)]$, which is the restriction on $S$ of the Picard stack $h^1/h^0(\R \bar \pi_*(\bar \F \otimes \bar \E^\vee \otimes \bar \omega[1]))$ \cite[Section 2 and Section~7]{BFantechi}. Analogously, the complex $[I_{|S} \rightarrow \Omega_{|S}]$ naturally induces an abelian cone stack $[\Spec\Sym(I_{|S})/\Spec\Sym(\Omega_{|S})]$, which is the restriction on $S$ of the relative intrinsic normal sheaf $\mathfrak N_{\S/\N} = h^1/h^0(\tau_{\geq -1}\Left_G)$. By \cite[2.6]{BFantechi} the perfect obstruction theory $\xi_0$ induces a closed embedding $\xi_0^\vee : [\Spec\Sym(I_{|S})/\Spec\Sym(\Omega_{|S})] \rightarrow [\Spec\Sym(\R^1 (\bar \pi_0)_*\bar K_0)/\Spec\Sym(\R^1 (\bar \pi_0)_*\bar M_0)]$. One can do this procedure on every local chart of $\N$ and check that the morphisms $\xi_0^\vee$ glue to a global morphism $\xi^\vee : \mathfrak N_{\S/\N} \rightarrow h^1/h^0(\R \bar \pi_*(\bar \F \otimes \bar \E^\vee \otimes \bar \omega[1]))$, which is a closed embedding by construction. Hence, by \cite[2.6]{BFantechi}, $\xi^\vee$ is induced by a morphism $\xi : \R \bar \pi_*(\bar \F \otimes \bar \E^\vee \otimes \bar \omega[1]) \longrightarrow \tau_{\geq -1}\Left_G$ which is a perfect obstruction theory for $G : \S \rightarrow \N$.
\end{proof}


\begin{coro}\label{cTanOb}
Let $n := (C \rightarrow \Spec k, F, E) \in \N(\KK)$ be a $\KK$-point of $\N$ (see~\ref{DstackN}) and let $s \in \S(\KK)$ be a $\KK$-point of $\S$ such that $G(s) = n$. Then the tangent space of $G : \S \rightarrow \N$ at $s$ is $\Hom(F, E)$ and an obstruction space of $G$ at $s$ is $\Ext^1(F,E)$.
\end{coro}

\begin{proof}
Define $P^\bullet := \LL s^*\R \bar \pi_*(\bar \F \otimes \bar \E^\vee \otimes \bar \omega[1]) \in D^{[-1,0]}_{\Coh}(\KK)$. By base change we have that $P^\bullet \cong \R\Gamma(C, F \otimes E^\vee \otimes \omega_C[1])$ where $\omega_C$ is the dualizing sheaf of $C$. Hence $\h^0((P^\bullet)^\vee) \cong \Hom(F, E)$ and $\h^1((P^\bullet)^\vee) \cong \Ext^1(F, E)$.
\end{proof}

\begin{prop}\label{Pmorph}
If the universal sheaf $\F$ of $\N$ is locally free, then there exists a canonical morphism 
\[
\R \bar \pi_*(\bar \F \otimes \bar \E^\vee \otimes \bar \omega[1]) \longrightarrow \Left_G 
\]
which induces the obstruction theory for $G : \S \rightarrow \N$.
\end{prop}

\begin{proof}
Since both $\F$ and $\bar \F$ are locally free sheaves ($\bar \F$ is the pullback of $\F$, see~\ref{NunivFamS} and~\ref{NunivFamN}), the abelian cones $p : \Spec\Sym(\F \otimes \E^\vee) \rightarrow \N'$ and $\bar p : \Spec\Sym(\bar \F \otimes \bar \E^\vee) \rightarrow \S'$ are vector bundles and, hence, their structure morphisms $p$ and $\bar p$ are smooth. The universal morphism $\phi : \bar \F \rightarrow \bar \E$ canonically induces a section $f : \S' \rightarrow \Spec\Sym(\bar \F \otimes \E^\vee)$ of $\bar p$. Denote by $G' : \S' \rightarrow \N'$ the pullback morphism of $G : \S \rightarrow \N$ via $\pi : \N' \rightarrow \N$, and by $v : \Spec\Sym(\bar \F \otimes \bar \E^\vee) \rightarrow \Spec\Sym(\F \otimes \E^\vee)$ the pullback morphism of $G : \S \rightarrow \N$ via $\pi \circ p : \Spec\Sym(\F \otimes \E^\vee) \rightarrow \N$. We have natural isomorphisms $\Left_{G'} \cong \LL\bar \pi^* \Left_G$ and $\Left_{\bar p} \cong \LL v^*\Left_p$.
Define $z := p \circ v$; then $G' = z \circ f$. Hence we have distinguished triangles
\begin{align*}
\LL v^*\Left_p \rightarrow \Left_z \rightarrow \Left_v \rightarrow  \LL v^*\Left_p[1] &\,, \\
\LL f^*\Left_z \rightarrow \Left_{G'} \rightarrow \Left_f \rightarrow \LL f^*\Left_z[1] &\,.
\end{align*}
From the first triangle we get a map $\Left_{\bar p} \rightarrow \Left_z$ and hence a map $\LL f^*\Left_{\bar p} \rightarrow \LL f^*\Left_z$; from the second triangle we get a map $\LL f^*\Left_z \rightarrow \LL \bar \pi^*\Left_G$. Composing them we obtain a map $\LL f^*\Left_{\bar p} \rightarrow \LL \bar \pi^*\Left_G$.
But $\Left_{\bar p} \cong\Omega_{\bar p} \cong \LL \bar p^*(\bar \F \otimes \bar \E^\vee)$ and, hence, $\LL f^*\Left_{\bar p} \cong \LL f^* \LL p^*(\bar \F \otimes \bar \E^\vee) \cong \bar \F \otimes \bar \E^\vee$ and we have a map
\[
\bar \F \otimes \bar \E^\vee \rightarrow \LL \bar \pi^* \Left_G \,.
\]
Now we use Grothendieck duality:
\begin{align*}
\Hom_{\D^b(\S')}(\bar \F \otimes \bar \E^\vee,\bar \pi^* \Left_G) 
			&\cong \Hom_{\D^b(\S')}(\bar \F \otimes \bar \E^\vee \otimes \bar \omega[1],\bar \pi^* \Left_G \otimes \, \bar \omega[1]) \\
			&\cong \Hom_{\D^b(\S')}(\bar \F \otimes \bar \E^\vee \otimes \bar \omega[1],\bar \pi^! \Left_G) \\
			&\cong \Hom_{\D^b(\S)}(\R \bar \pi_*(\bar \F \otimes \bar \E^\vee \otimes \bar \omega[1]), \Left_G) \,.
\end{align*}
Hence, the morphism $\bar \F \otimes \bar \E^\vee \rightarrow \LL \bar \pi^* \Left_G$ naturally induces a morphism 
\[
\R \bar \pi_*(\bar \F \otimes \bar \E^\vee \otimes \bar \omega[1]) \rightarrow \Left_G \,. \qedhere
\]
\end{proof}

\section{Rigidification}\label{sRigid}
Let $X$ be an algebraic stack, let $H$ be a separated group scheme and assume that for any affine scheme $T$ and any $x \in X(T)$ there is an injective morphism of groups $H(T) \rightarrow \Aut(x)$ which is compatible with pullbacks. As shown in~\cite[Section~5.1]{acv}, this implies the existence of a canonical algebraic stack $X^H$ and a canonical morphism $X \rightarrow X^H$ which makes  $X$ into a gerbe over $X^H$ banded by $H$. The morphism $X \rightarrow X^H$ is called \textit{the rigidification of $X$ along $H$}.

In this section we describe the rigidification of the stack $\S$ with respect to the multiplication by scalars and we prove that the obstruction theory for the forgetful morphism $G : \S \rightarrow \N$ descends to the rigidification (the stacks $\S$ and $\N$ and the morphism $G : \S \rightarrow \N$ are defined in Section~\ref{Sstacks}; we constructed a relative obstruction theory for $G : \S \rightarrow \N$ is Section~\ref{SobTh}).

\begin{notation}
We use the notation $\GG$ to denote the multiplicative group scheme $\GG_m = \Spec \C[t, t^{-1}]$ and $\O(\GG)$ to denote the space of global sections $\H^0(\GG,\O_\GG)$.
\end{notation}

\begin{lemma}
Let $T$ be an affine scheme and let $s := (C \rightarrow T, F \rightarrow E) \in \S(T)$ be a family of generalized coherent systems over $T$ (see~\ref{Dgcs}). Then there is a canonical injective morphism of groups $\GG(T) \rightarrow \Aut(s)$ which is compatible with pullbacks.
\end{lemma}

\begin{proof}
Assume $T = \Spec R$; then $\GG(T)$ is the group of invertible elements of $R$, $R^\times$. But any $r \in R^\times$ canonically induces an automorphism of $s$ by multiplication. The induced group map $\GG(T) \rightarrow \Aut(s)$ is injective because $F$ and $E$ are flat $R$-modules.
\end{proof}

\begin{rem}
With an analogous argument one can check that for any $n \in \N(T)$ there is a canonical injection $\GG(T) \rightarrow \Aut(n)$ which is compatible with pullbacks.
\end{rem}

\begin{prop}\label{PrigMor}
Let $\S \rightarrow \S^\GG$ be the rigidification of $\S$ along $\GG$ and let $\N \rightarrow \N^\GG$ be the rigidification of $\N$ along $\GG$. Then there exists a unique morphism (up to unique 2-isomorphism) $\tilde G : \S^\GG \rightarrow \N^\GG$ such that the following diagram is 2-cartesian:
\[
\begin{tikzcd}
\S \arrow[r] \arrow[d, "G"] & \S^\GG \arrow[d, "\tilde G"] \\
\N \arrow[r] & \N^\GG
\end{tikzcd} \,.
\]
\end{prop}

\begin{proof}
That is a consequence of the universal property of the rigidification $\S \rightarrow \S^\GG$.
\end{proof}

\begin{lemma}\label{Lhi}
Let $n := (C \rightarrow \Spec k, F, E) \in \N(k)$ and let $\lambda \in \GG(k)$ be a (nonzero) scalar. Then the automorphism induced on $\Ext^i(F, E)$ (for $i = 0,1$) by acting simultaneously on $E$ and on $F$ with the scalar $\lambda$ is the identity.
\end{lemma}

\begin{proof}
Since the functor $\Ext^i$ is contravariant in the first variable and covariant in the second, the scalar automorphism $\lambda$ applied to the first variable acts as $\lambda^{-1}$, and applied to the second variable it acts as $\lambda$. 
\end{proof}

\begin{prop}\label{POstrScend}
Let $X$ be an algebraic stack and let $p : \X \rightarrow X$ be a gerbe banded by $\GG$. Define
\begin{align*}
\mathcal T := \left\{\E^\bullet \in \D^-_{\Coh}(\X) \left|
		\begin{aligned}
 		p^*p_* \h^i(\E^\bullet) \cong \h^i(\E^\bullet)
		\end{aligned}
		\right. \right\} \,.
\end{align*}
Then the derived functor $\LL p^* : \D^-_{\Coh}(X) \rightarrow \D^-_{\Coh}(\X)$ induces an equivalence of categories $\D^-_{\Coh}(X) \cong \mathcal T$.
\end{prop}

\begin{proof}
By definition of gerbe the morphism $p : \X \rightarrow X$ is flat, hence the derived functor $\LL p^* : \D^-_{\Coh}(X) \rightarrow \D^-_{\Coh}(\X)$ is just the ordinary pullback $p^*$. One can check that also $p_*$ is exact. 
Hence for any $\E^\bullet \in \D^-_{\Coh}(\X)$ we have that $\h^i(p^*p_*\E^\bullet) \cong p^*p_*(\h^i(\E^\bullet))$. Therefore if $\E^\bullet \in \mathcal T$ then the canonical morphism $p^*p_*\E^\bullet \rightarrow \E^\bullet$ is an isomorphism in $\D^-_{\Coh}(\X)$. On the other hand one can check that for any $E \in \Qcoh(X)$ the canonical morphism $E \rightarrow p_*p^*E$ is an isomorphism. Hence for any complex $E^\bullet \in \D^-_{\Coh}(X)$ the canonical morphism $E^\bullet \rightarrow p_*p^*E^\bullet$ is an isomorphism.
\end{proof}

\begin{coro}\label{Cdescends}
The obstruction theory $E^\bullet \rightarrow \tau_{\geq -1}\Left_G$ (introduced in \ref{PobTh}) canonically induces a perfect obstruction theory for the rigidified morphism $\tilde G : \S^\GG \rightarrow \N^\GG$ (see~\ref{PrigMor}).
\end{coro}

\begin{proof}
Recall that $E^\bullet := \R\pi_*(\bar \F \otimes \bar \E^\vee \otimes \bar \omega[1])$. Let $p : \S \rightarrow \S^\GG$ be the rigidification of $\S$ along $\GG$. By Proposition~\ref{POstrScend}, Lemma~\ref{Lhi} and Nakayama's Lemma there exists a unique $\tilde E^\bullet \in \D^{[-1,0]}_{\Coh}(\S^\GG)$ such that $E^\bullet \cong p^* \tilde E^\bullet$. Moreover $\Left_G \cong p^*\Left_{\tilde G}\,,$ since the morphism $p$ is smooth. Hence the morphism $E^\bullet \rightarrow \tau_{\geq -1}\Left_G$ induces a morphism $\tilde E^\bullet \rightarrow \tau_{\geq -1}\Left_{\tilde G}$, again by Proposition~\ref{POstrScend}. One can check that $\tilde E^\bullet \rightarrow \tau_{\geq -1}\Left_{\tilde G}$ is a perfect obstruction theory for $\tilde G : \S^\GG \rightarrow \N^\GG$.
\end{proof}


\section{Applications}

In this section we prove that the relative obstruction theory defined in Section~\ref{SobTh} induces a perfect obstruction theory for the moduli spaces of $\alpha$-stable coherent systems.

Throughout this section we will assume that the universal sheaf $\F$ of $\N$ (see~\ref{NunivFamS} and~\ref{NunivFamN}) is locally free. Hence the universal sheaf $\bar \F$ of $\S$ is locally free, too.

\begin{defi}
Let $s \in \S(T)$ and let $m \in \M(T)$ be the image of $s$ via the forgetful morphism $F : \S \rightarrow \M$. Let
\[
\Aut(s/\id_m) := \{\psi \in \Aut(s) \,|\, F(\psi) = \id_m\} \,.
\]
We say that $s$ is \textit{simple} if $\iota_s : \GG(T) \rightarrow \Aut(s/\id_m)$ is an isomorphism.
\end{defi}

\begin{notation}
The symbol $\S_{smp}$ denotes the stack of simple generalized coherent systems. Note that the canonical morphism $\S_{smp} \rightarrow \S$ is an open embedding.

With an abuse of notation, we still use the letter $G$ to denote the morphism $G : \S_{smp} \rightarrow \N$, which is the restriction of the forgetful morphism $\S \rightarrow \N$ to $\S_{smp} \,$.
\end{notation}

\begin{rem}
The forgetful morphism $\N \rightarrow \M$ (see~\ref{DstackN}) factors through the rigidification $\N^\GG$ (see~\ref{PrigMor}). Notice that the morphism $\S^\GG_{smp} \rightarrow \M$ is representable. Moreover the morphism $\N^\GG \rightarrow \M$ is smooth, since $\N \rightarrow \M$ is smooth.
\end{rem}

\begin{prop}
The obstruction theory $E^\bullet \rightarrow \tau_{\geq -1}\Left_G$ (introduced in \ref{PobTh}) canonically induces a perfect obstruction theory for the morphism $\S^\GG_{smp} \rightarrow \M$.
\end{prop}

\begin{proof}
Let $q : \N^\GG \rightarrow \M$; since $q$ is smooth, the complex $\Left_q$ is perfect. Denote by $F : \S_{smp}^\GG \rightarrow \M$ the composition $q \circ \tilde G$ (see~\ref{PrigMor} for the definition of $\tilde G : \S^\GG \rightarrow \N^\GG$). By the properties of the cotangent complex we have a distinguished triangle
\[
\LL \tilde G^* \Left_q \rightarrow \Left_F \rightarrow \Left_{\tilde G} \rightarrow \tilde \LL G^* \Left_q[1] \,.
\]
By Proposition~\ref{Pmorph} the obstruction theory $E^\bullet \rightarrow \tau_{\geq -1}\Left_G$ is induced by a canonical morphism $E^\bullet \rightarrow \Left_G$. This morphism descends to the rigidification, as described in Corollary~\ref{Cdescends}. Hence we get a canonical morphism $\tilde E^\bullet \rightarrow \Left_{\tilde G}$ which induces the perfect obstruction theory for $\tilde G : \S^\GG \rightarrow \N^\GG$. Let $E'$ denote the mapping cone of such morphism, shifted by $-1$. By the axioms of the triangulated categories we obtain a morphism $E' \rightarrow \Left_F$ and therefore a morphism $E' \rightarrow \tau_{\geq -1}\Left_F$. One can check that it induces a perfect obstruction theory for the morphism $\S_{smp}^\GG \rightarrow \M$. 
\end{proof}

\begin{rem}\label{rStackCS}
Fix a triple of integers $(n, d, k)$. Let $\BGL_k$ be the quotient stack $[\Spec\KK/\GL_k]$. Let $\VecM^{n,d}$ be the open substack of $\VecM$ whose $T$-points are pairs $(C \rightarrow T, E)$, where $C \rightarrow T$ is an $\M$-family of curves and $E \in \Coh(C)$ is locally free of rank $n$ and degree $d$ (see~\ref{DstackN}). There is a natural morphism $\BGL_k \times \VecM^{n,d} \rightarrow \N$ such that for any scheme $T$ we have
\begin{align*}
\BGL_k(T) \times \VecM^{n,d}(T) &\rightarrow \N(T) \,, \\
(V, (p: C \rightarrow T, E)) &\mapsto (p: C \rightarrow T, p^*V, E) \,.
\end{align*}
Let $\S_1 := \S \times_\N (\BGL_k \times \VecM^{n,k})$. Define a stack $\SCS$ as the open substack of $\S_1$ whose $T$-points are families $(V, (p: C \rightarrow T, E), \varphi : p^*V \rightarrow E)$ where $(V, (p: C \rightarrow T, E)) \in \BGL_k(T) \times \VecM^{n,d}(T)$ and $p^*V \rightarrow E$ is an injective morphism.
\end{rem}

\begin{defi}
We call the stack $\SCS$ introduced in Remark~\ref{rStackCS} the \textit{moduli stack of coherent systems}.
\end{defi}

\begin{coro}\label{CObCohSys}
For any smooth, projective, genus $g$ curve $C$ and for any triple $(n,d,k)$ the moduli space of simple coherent systems of type $(n,d,k)$ has a perfect obstruction theory of rank
\[
\beta := n^2(g - 1) + 1 - k(k - d + n(g - 1)) \,.
\]
\end{coro}

\begin{proof}
To prove this result we may assume that $\M = \Spec \KK$ and $\M' = C$ (recall that $\M'$ is a family of projective Gorenstein curves over $\M$, as described in Section~\ref{Sstacks}).

Let $\SCS$ be the moduli stack of coherent systems and consider the natural morphism $F : (\SCS)_{smp}^\GG \rightarrow (\VecM^{n,d} \times \BGL_k)^\GG$ (compare Remark~\ref{rStackCS} and Section~\ref{sRigid}). The obstruction theory constructed in Section~\ref{SobTh} naturally induces a perfect relative obstruction theory for the morphism $F$; denote by $r$ its rank. Denote by $\delta$ the dimension of $(\VecM^{n,d} \times \BGL_k)^\GG$.
We need to check that $\beta = r + \delta$.

Since the relative obstruction theory for $F$ is perfect, we can compute its rank at any $\KK$-point. Fix $(V, (C \rightarrow \Spec \KK, E), \varphi : V \otimes \O_C \rightarrow E) \in \SCS(\KK)$; by hypothesis $E$ is a vector bundle of rank $n$ and degree $d$, while $V$ is a vector space of dimension $k$. The rank of an obstruction theory at a point is given by the dimension of the tangent space minus the dimension of the obstruction space at that point, so we have that
\[
r = \dim \Hom(V \otimes \O_C, E) - \dim \Ext^1(V \otimes \O_C, E) \,,
\]
compare Corollary~\ref{cTanOb}. Hence
\[
r = \chi(E^{\oplus k}) = k(d + n(1 -g)) \,.
\]
On the other hand
\[
\delta := \dim (\VecM^{n.d} \times \BGL_k)^\GG = n^2(g-1) - k^2 + 1 \,.
\]
By comparison we deduce that $r + \delta = \beta$.
\end{proof}

Since $\alpha$-stable coherent systems are simple, our computation provides a perfect obstruction theory of rank equal to $\beta$ for every moduli space of $\alpha$-stable coherent systems.

\begin{coro}\label{CobStabCohSys}
Fix $\alpha \in \R$. Let $C$ be a smooth, projective, genus $g$ curve and let $(n,d,k)$ be a suitable triple of positive integers. Let
\[
\beta := \beta(n,d,k) = n^2(g - 1) + 1 - k(k - d + n(g - 1))
\] be the Brill Noether number \cite[2.7]{CohBrNoet}. Then the moduli space of $\alpha$-stable coherent systems of type $(n,d,k)$ has a perfect obstruction theory of rank~$\beta$. \qed
\end{coro}

\addcontentsline{toc}{section}{Bibliography}

\end{document}